\documentclass[11pt]{amsart}
\usepackage{amssymb}
\usepackage{verbatim}

\newtheorem{lemma}{Lemma} [section]
\newtheorem{thm}[lemma]{Theorem}
\newtheorem{cor}[lemma]{Corollary}
\newtheorem{prop}[lemma]{Proposition}
\newtheorem{definition}[lemma]{Definition}
\newtheorem{claim}[lemma]{Claim}
\newtheorem*{case1}{Case 1}
\newtheorem*{case2}{Case 2}
\newtheorem*{case3}{Case 3}
\newtheorem*{case4}{Case 4}
\newtheorem*{case4a}{Case 4a}
\newtheorem*{case4b}{Case 4b}

\theoremstyle{remark}
\newtheorem*{remark}{Remark}

\theoremstyle{definition}

\numberwithin{equation}{section}

\DeclareMathOperator{\Aut}{{Aut}}
\DeclareMathOperator{\cyc}{{cycles}}
\DeclareMathOperator{\Core}{{core}}
\DeclareMathOperator{\Sym}{{Sym}}
\DeclareMathOperator{\lcm}{{lcm}}
\DeclareMathOperator{\charp}{{char}}
\newcommand{\F}{{\mathbb F}}
\newcommand{\C}{{\mathbb C}}
\newcommand{\cC}{{\mathcal C}}
\newcommand{\cI}{{\mathcal I}}
\newcommand{\Line}{{\mathbb P}^1}
\newcommand{\meet}{\wedge}
\newcommand{\join}{\vee}

\begin{document}



\title[Analogues of the Jordan--H\"older theorem]
{Analogues of the Jordan--H\"older theorem for transitive $G$-sets}

\author{Greg Kuperberg}
\address{Department of Mathematics,
  University of California,
  Davis, CA 95616
}
\email{greg@math.ucdavis.edu}
\urladdr{www.math.ucdavis.edu/$\sim$greg/}

\author{Michael E. Zieve}
\address{
  Center for Communications Research,
  805 Bunn Drive,
  Princeton, NJ 08540
}
\email{zieve@math.rutgers.edu}
\urladdr{www.math.rutgers.edu/$\sim$zieve/}
\thanks{The authors thank Richard Lyons and Peter Neumann for
valuable correspondence.  The first author was supported by the
National Science Foundation under Grant No. 0606795}


\begin{abstract}
Let $G$ be a transitive group of permutations of a finite set
$\Omega$, and suppose that some element of $G$ has at most two
orbits on $\Omega$.  We prove that any two maximal chains of
groups between $G$ and a point-stabilizer of $G$ have the same
length, and the same sequence of relative indices between
consecutive groups (up to permutation).  We also deduce the same
conclusion when $G$ has a transitive quasi-Hamiltonian subgroup.
\end{abstract}


\maketitle


\section{Introduction}

One of the few mistakes in Jordan's classic 
\emph{Trait\'e des substitutions} is an assertion that would
now be called the Jordan--H\"older theorem for transitive $G$-sets.
Specifically, he asserted in \cite[\S 51, p.~38]{J}
that, for every subgroup $H$ of a finite group $G$, the pair
$(G,H)$ has the following property:

\begin{definition}\label{Jordan}
Let $H$ be a finite-index subgroup of a group $G$.
We say the pair $(G,H)$ has the \emph{Jordan property} when the
following holds: if $A_1 \subsetneqq\dots\subsetneqq A_a$ and
$B_1\subsetneqq\dots\subsetneqq B_b$ are maximal chains of groups
between $H$ and $G$, then $a=b$ and the sequence
$([A_2:A_1],\dots,[A_a:A_{a-1}])$ is a permutation of
$([B_2:B_1],\dots,[B_b:B_{b-1}])$.
\end{definition}

Jordan realized his mistake soon after publication, and retracted
his assertion \cite{J2}; the smallest counterexample is $(A_4,1)$,
in which $1<C_2<V_4<A_4$ and $1<C_3<A_4$ are maximal chains having
distinct lengths.  A half-century later, Ritt discovered that
$(G,H)$ has the Jordan property if $G=HI$ for some finite cyclic
subgroup $I$ of $G$; this was a key ingredient in Ritt's work on
functional equations, yielding a fundamental invariant of
functional decompositions of a complex polynomial.  After another
sixty years, M\"uller showed that $(G,H)$ has the Jordan property
if $G=HI$ for some finite abelian subgroup $I$ of $G$.  We will
give a simpler proof of M\"uller's result, while also extending
it to a larger class of groups:

\begin{prop}\label{abelianintro}
If $H$ is a subgroup of a group $G$, and $G=HI$
for some finite subgroup $I$ of $G$ such that
\begin{equation}\tag{$*$}
\text{any two subgroups $I_1,I_2$ of $I$ are permutable (i.e.,
 $I_1I_2=I_2I_1$),}
\end{equation}
then $(G,H)$ has the Jordan property.
\end{prop}

Note that $I_1I_2=I_2I_1$ if and only if $I_1I_2$ is a group, or
equivalently $\#\langle I_1,I_2\rangle=(\#I_1)[I_2:I_1\cap I_2]$.
Abelian groups $I$ satisfy $(*)$, as do Hamiltonian groups (i.e.,
nonabelian groups with no nonnormal subgroups).
As shown by Dedekind \cite{D}, the finite Hamiltonian groups
consist of the direct products of the order-$8$ quaternion group
with an abelian group containing no elements of order $4$.  There
is a similar but less known classification of finite groups
satisfying $(*)$ (called quasi-Hamiltonian groups), due to Pic
\cite{Pic}.

Our main result says that $(G,H)$ has the Jordan property if
some cyclic subgroup $I$ of $G$ has two orbits on the set of
cosets of $H$ in $G$:

\begin{thm}\label{mainthm}
Let $H$ be a subgroup of a group $G$, and let $I$
be a finite cyclic subgroup of $G$ such that $G=IH\cup IgH$ for
some $g\in G$.  Then $(G,H)$ has the Jordan property.
\end{thm}

By means of a now-standard inductive argument due to Netto
(cf.\ Lemma~\ref{dl}),
this result is a consequence of the following:

\begin{thm}\label{diamond}
Let $G$ be a transitive group of permutations of a finite set
$\Omega$, fix $\omega\in\Omega$, and suppose some element of
$G$ has at most two cycles on $\Omega$.  If $A$ and $B$ are
distinct maximal subgroups of $G$ for which $H:=A\cap B$
contains $G_{\omega}$, then one of the following holds:
\begin{enumerate}
\item[(\thethm.1)] $H$  is a maximal subgroup of both $A$
  and $B$, and $G=AB$; or
\item[(\thethm.2)] $G/\Core_G(H)$ is dihedral of order $2r$,
 with $r$ prime, and both $A$ and $B$ have order $2$
 images in $G/\Core_G(H)$.
\end{enumerate}
\end{thm}

Here, as usual, $\Core_G(H)$ is the maximal normal subgroup
of $G$ contained in $H$.  Possibility (\ref{diamond}.2)
illustrates a phenomenon not arising in
Proposition~\ref{abelianintro}: if $G$ contains a transitive
quasi-Hamiltonian subgroup, then (\ref{diamond}.1) holds.

We will give examples showing that Assertion~\ref{Jordan} need
not hold for transitive groups $G$ containing an element with
three cycles, and likewise for transitive groups $G$ containing
an abelian subgroup with two orbits.  However, in some sense the
counterexamples in both situations appear to be bounded, so it
may be possible to classify the counterexamples to
Theorem~\ref{diamond} in these more general situations.

Our work was motivated by geometric applications.  Specifically,
Theorem~\ref{mainthm} has the following consequence:

\begin{cor}\label{curves}
Let $C$ and $D$ be smooth, projective, geometrically irreducible
curves over a field $K$, and let $f:C\to D$ be a nonconstant
separable rational map defined over $K$.  Suppose that some place
of $D$ is tamely ramified in $f$, and lies under at most two
places of $C$.  Let $C\to A_1\to A_2\to\dots\to D$ and
$C\to B_1\to B_2\to\dots\to D$ be maximal decompositions of $f$
into rational maps of degree more than $1$.  Then these
decompositions have the same length, and (up to permutation) the
same sequence of degrees of the involved indecomposable maps.
\end{cor}

The simplest case $C=D=\Line$ is already interesting: there $f$
is essentially a Laurent polynomial.

\begin{cor}\label{Laurent}
Let $f\in K[x,x^{-1}]$ be a Laurent polynomial over a field $K$,
and assume that neither $x=0$ nor $x=\infty$ is a pole of $f$ of
order divisible by $\charp(K)$.  Write
$f=a_1\circ\dots\circ a_r = b_1\circ\dots\circ b_s$ where
$a_i,b_j\in K(x)$ are indecomposable and have degree more than $1$.
Then $r=s$, and the sequence $(\deg(a_1),\dots,\deg(a_r))$ is a
permutation of $(\deg(b_1),\dots,\deg(b_s))$.
\end{cor}

In case $K=\C$, this result was proved in the recent paper \cite{Z};
subsequently another proof was given in \cite{P}.  Conversely, the
$K=\C$ case of Corollary~\ref{Laurent} is equivalent to some special
cases of Theorem~\ref{mainthm}, in view of Riemann's existence
theorem and knowledge of the fundamental group of the punctured
sphere.  Specifically, the $K=\C$ case of Corollary~\ref{Laurent}
is equivalent to Theorem~\ref{mainthm} for groups $G$ having
generators $g_1,\dots,g_k$ where $g_1$ has two cycles and
$g_1g_2\dots g_k=1$ and $2\#S-2=\sum_i (\#S - \#\cyc(g_i))$.

The proofs in \cite{Z} and \cite{P} relied on various
algebro-geometric calculations, which were much more complicated
than the proof in the present paper.  On the other hand, in case
$K=\C$, those papers obtained quite precise information about the
shape of the indecomposable rational functions $a_i$ and $b_j$
occurring in Corollary~\ref{Laurent}.  This precise information
relies on the fact that $C$ and $D$ have genus zero, so it is
not surprising that geometry is required for the proof.
However, since Corollary~\ref{Laurent} has a group-theoretic
interpretation, this result seemed to merit a group-theoretic
proof, which we produce in the present paper.  As a bonus, this
group-theoretic proof implies Corollary~\ref{Laurent} even for
fields of positive characteristic, and also implies
Corollary~\ref{curves}; essentially, the group-theoretic proof uses
just the ramification data at the special point, whereas the
geometric proofs used the ramification data at all points, hence
required stronger hypotheses.

Ritt proved Corollary~\ref{Laurent} in case $K=\C$ and $f$ is a
polynomial \cite{Ritt}.  He also determined the possibilities for
the $a_i$'s and $b_j$'s, and his results have been applied to a
wide range of topics
(cf.\ \cite{BWZ,BT,GTZ,GTZ2,MP,P1,P2,PRY,Zannier}, among others).
Here we mention just the most recent application, from \cite{GTZ2}:

\begin{thm}[Ghioca--Tucker--Zieve]
Let $x_0,y_0\in\C$ and $f,g\in\C[x]$ satisfy $\deg(f),\deg(g)\ne 1$.
If the orbits
\[
\{x_0, f(x_0), f(f(x_0)), \dots\}\quad\text{ and }\quad
\{y_0, g(y_0), g(g(y_0)), \dots\}
\]
have infinite intersection, then $f$ and $g$ have a common iterate.
\end{thm}

It would be of great interest to extend this result to Laurent
polynomials or more general rational functions.  We suspect the
group theoretic perspective of the present paper may be useful
in this endeavor.

We now summarize the contents of this paper.  In the next section
we present a version of the diamond lemma, which we use in
Section~\ref{secab} to prove Proposition~\ref{abelianintro}.
In Section~\ref{Gsets} we record the terminology of $G$-sets
needed for the proofs of our main results, and in Section~\ref{secab}
we prove Theorems~\ref{mainthm} and \ref{diamond}.
We conclude in the final section with some examples and speculations.


\section{Diamond lemma}

In order to prove that a pair of groups has the Jordan property,
it suffices to consider pairs of chains whose common least element
is the intersection of their two second-largest elements.  This
generalizes an argument due to Netto~\cite{N}.

\begin{lemma}\label{dl} Let $\cC$ be a set of
pairs $(G,H)$ of a group $G$ and a finite-index subgroup $H$, and
suppose that if $(G,H)\in\cC$ and $H<G_0<G$ then $\cC$ contains both
$(G_0,H)$ and $(G,G_0)$.
Suppose further that, if $(G,H)\in\cC$ and $A,B$ are distinct maximal
subgroups of $G$ with $H=A\cap B$, then there exist maximal chains
of groups $G > A > \dots > H$ and $G > B > \dots > H$ such that the
sequences of indices in the two chains are the same up to
permutation.  It follows that every pair in $\cC$ has the Jordan
property.
\end{lemma}

\begin{proof}
Pick $(G,H)\in\cC$ with $[G:H]=d$, and suppose the result holds for
pairs with smaller index.  Let $A_0 < A_1 < \dots < A_r$ and
$B_0<\dots<B_s$ be maximal chains of groups between $H$ and $G$.
Then $A_r=B_s=G$, and $A_{r-1}$ and $B_{s-1}$ are maximal subgroups
of $G$.  If $A_{r-1}=B_{s-1}$ then the result follows from the
corresponding result for $(A_{r-1},H)$.  So assume
$A_{r-1}\ne B_{s-1}$, and let $K=A_{r-1}\cap B_{s-1}$.
The hypothesis for the case $(G,K)$ is that there are maximal
chains $G>A_{r-1}>\dots>K$ and $G>B_{s-1}>\dots>K$ for which
the sequences of indices are the same up to permutation.
Pick a maximal chain $K>\dots>H$, and concatenate with the previous
chains to get maximal chains
$G>A_{r-1}>\dots>H$ and $G>B_{s-1}>\dots>H$ with the same multiset
of indices.  The result for $(A_{r-1},H)$ implies that the
multiset of indices for the first chain equals that for the
$A_i$ chain.  The result for $(B_{s-1},H)$ implies that the multiset
for the second chain equals that for the $B_j$ chain.
Thus, the $A_i$ chain has the same multiset of indices as does
the $B_j$ chain.
\end{proof}


\section{Groups with a transitive quasi-Hamiltonian subgroup}
\label{secab}

In this section we prove Proposition~\ref{abelianintro}.
In light of Lemma~\ref{dl}, it suffices to prove the following:

\begin{prop}\label{abelian}
Let $H$ be a subgroup of a group $G$, and suppose that $G=HI$
for some finite subgroup $I$ of $G$ satisfying $(*)$.
If $A,B$ are distinct maximal subgroups of $G$ with
$A\cap B=H$, then $H$ is maximal in both $A$ and $B$, and
$G=AB$.
\end{prop}

We begin with a general lemma on subgroups of groups of this form.

\begin{lemma}\label{tran}
Let $H$ and $I$ be subgroups of a group $G$ such that $G=HI$.
Then $W\mapsto W\cap I$ is an
injective map from the set of groups between $G$ and $H$ to the set
of groups between $I$ and $I\cap H$.  This map preserves indices
between pairs of groups, and its image is closed under intersections
and joins (where the join of two groups is the group they generate).
\end{lemma}

\begin{proof}
Let $W$ be a group between $H$ and $G$, so $W$ is a union of cosets
of $H$, and each such coset contains an element of $I$, whence
$W=H(W\cap I)$.  Thus $[W:H]=[W\cap I:H\cap I]$.  If $A,B$ are
groups between $H$ and $G$, write $A=HC$ and $B=HD$ with
$C,D\le I$.  Then $A\cap B=HE$ with 
\[
E=I\cap (A\cap B) = (I\cap A)\cap(I\cap B)=C\cap D.
\]
Also $H\langle C,D\rangle = \langle C,D\rangle H$,
 and hence equals $\langle HC,HD\rangle$.  
\end{proof}

In view of Lemma~\ref{tran},  Proposition~\ref{abelian} is a
consequence of the following result:

\begin{lemma}
Let $I$ be a finite group satisfying $(*)$, let $H$ be a subgroup of
$I$, and let $\cI$ be a set of groups between $H$ and $I$ which
includes both $H$ and $I$, and which is closed under intersections
and joins.  If $A,B$ are maximal subgroups of $I$ with $A\cap B=H$,
then $H$ is maximal in both $A$ and $B$, and $I=AB$.
\end{lemma}

\begin{proof}
Let $A,B$ be distinct
maximal elements of $\cI\setminus\{I\}$.  Then $I=AB$, so
$[I:B]=[A:A\cap B]$.  Also if $J\in\cI$ is strictly between
$A\cap B$ and $A$, then $[BJ:B]=[J:A\cap B]$ so $BJ$
is strictly between $B$ and $I$, contradicting maximality.
Therefore both chains $A\cap B < A < I$ and
$A\cap B < B < I$ are maximal in $\cI$.
\end{proof}

This completes the proof of Proposition~\ref{abelian}, and so of
Proposition~\ref{abelianintro}.

\begin{remark}
M\"uller's proof of Proposition~\ref{abelian} in the case of
abelian groups $I$ is substantially more complicated than the
one given above.
\end{remark}


\section{$G$-sets}
\label{Gsets}

In this section we record some notation and terminology involving
$G$-sets.  For standard notions of $G$-sets, we refer to \cite{NST},
which we follow when possible; below we give the details of some
notions that are not universally used.

Given a group $G$, by a $G$-set we mean a nonempty set $\Omega$
together with a homomorphism $\rho:G\to\Sym(\Omega)$.  For $g\in G$
and $\omega\in\Omega$, we write $\omega^g$ for $(\rho(g))(\omega)$.
An equivalence relation $\phi$ on the $G$-set $\Omega$ is said to be
$G$-invariant if
\[
\alpha\equiv\beta\pmod\phi\Rightarrow \alpha^g\equiv\beta^g
\pmod\phi\quad\text{ for every $g\in G$.}
\]
Such an equivalence relation is called a \emph{congruence} on
$\Omega$.  If $\phi$ is a congruence on $\Omega$, then the action
of $G$ on $\Omega$ naturally induces an action of $G$ on the set
$\Omega/\phi$ of $\phi$-equivalence classes on $\Omega$; the $G$-set
$\Omega/\phi$ is called the quotient of the $G$-set $\Omega$ by
the congruence $\phi$.  The notions of transitive $G$-sets,
homomorphisms of $G$-sets, isomorphisms of $G$-sets, and direct
products of $G$-sets are defined in the usual manner.

Let $\phi$ and $\psi$ be two congruences on a $G$-set $\Omega$.
We say that $\phi$ coarsens $\psi$ (or equivalently, $\psi$ refines
$\phi$) if every $\phi$-equivalence class is a union of
$\psi$-equivalence classes.  We denote the coarsest common
refinement of $\phi$ and $\psi$ by $\phi\meet \psi$; thus, each
$\phi\meet \psi$-equivalence class is the intersection of a
$\phi$-equivalence class and a $\psi$-equivalence class.  We denote
the finest common coarsening of $\phi$ and $\psi$ by
$\phi\join \psi$; each $\phi\join \psi$-equivalence class is a union
of $\phi$-equivalence classes, and also a union of
$\psi$-equivalence classes.
We write $\phi^\psi$ for the coarsening of $\phi$ in which two
$\phi$-equivalence classes are $\phi^\psi$-equivalent if they
nontrivially intersect the same $\psi$-equivalence classes.
We emphasize that $\phi\meet\psi$, $\phi\join\psi$, and
$\phi^\psi$ are congruences on $\Omega$.

Any $G$-set $\Omega$ comes equipped with two trivial congruences:
the \emph{trivial coarse congruence}, in which $\Omega$ is itself an
equivalence class; and the \emph{trivial fine congruence}, in which
every equivalence class contains a single element.  Any congruence
besides these two is said to be \emph{nontrivial}.
We usually identify $\Omega$ with its trivial coarse congruence.
We say a transitive $G$-set $\Omega$ is \emph{primitive} if it
admits no nontrivial congruences.

Let $\phi$ and $\psi$ be congruences on a transitive $G$-set
$\Omega$.  Then $\Omega/\phi$ and $\Omega/\psi$ are transitive,
so any two $\psi$-equivalence classes have the same size, and
likewise for any two $\phi$-equivalence classes.
If $\phi$ coarsens $\psi$, then every $\phi$-equivalence class
consists of the same number $k$ of $\psi$-equivalence classes;
we call $k$ the \emph{index} of $\psi$ in $\phi$, and denote this
by $[\phi:\psi]$.

The following lemma is routine.

\begin{lemma}\label{gpset}
Let $\Omega$ be a transitive $G$-set, and pick
$\omega\in\Omega$.  Define a map $\theta$ from the set of
congruences on $\Omega$ to the set of groups between $G$ and
$G_\omega$ as follows: let $\theta(\phi)$ be the stabilizer of
the image of $\omega$ in the $G$-set $\Omega/\phi$.  Then
\begin{itemize}
\item $\theta$ is a bijection;
\item $\phi$ coarsens $\psi$ if and only if $\theta(\psi)$ is
 a subgroup of $\theta(\phi)$; and
\item if $\phi$ coarsens $\psi$ then
 $[\phi:\psi]=[\theta(\phi):\theta(\psi)]$.
\end{itemize}
\end{lemma}

Note that the final two assertions say that $\theta$ is
order-preserving and index-preserving.

Finally, we remark that the proof in the previous section can
be translated to the languate of $G$-sets.  Specifically,
let $I$ be a finite quasi-Hamiltonian subgroup of a group $G$,
and let $A$ and $B$ be distinct maximal subgroups of $G$ for which
$H:=A\cap B$ satisfies $G=HI$.  Let $\Omega$ be the $G$-set of
left cosets of $H$ in $G$, and let $\phi$ and $\psi$ be the
congruences on $\Omega$ corresponding (via Lemma~\ref{gpset}) to
$A$ and $B$.  Then, for $I_1,I_2\le I$, the property
$I_1I_2=I_2I_1$ can be restated as the condition that
$\phi\join\psi$ should be the full coset space $I/I_1I_2$.
Thus, the proof in the previous section can be stated either in
terms of groups or congruences, with little conceptual difference.
On the other hand, in the next section we will find the congruence
viewpoint to be more suitable for the problem at hand.


\section{Transitive groups with an element having at most two cycles}
\label{proofsec}

In this section we prove Theorem~\ref{diamond}; by Lemma~\ref{dl},
this implies Theorem~\ref{mainthm}.

Given a group $G$ and a finite-index subgroup $H$, let $\Omega$ be
the transitive $G$-set of left cosets of $H$ in $G$.  Suppose
$g\in G$ has at most two orbits (= cycles) on $\Omega$.
Let $\phi$ and $\psi$ be nontrivial congruences on $\Omega$.
Suppose further that $\phi$ and $\psi$ have no nontrivial common
refinement, so any $\phi$-equivalence class intersects any 
$\psi$-equivalence class in at most one element -- thus $\Omega$
embeds (as a $G$-set) into $\Omega/\phi \times \Omega/\psi$.
Suppose also that $\phi$ and $\psi$ have no nontrivial common
coarsening.  \emph{We maintain the above notation throughout
this section.}

By Lemmas~\ref{dl} and \ref{gpset}, to prove
Theorem~\ref{diamond} it suffices to show:
if $\Omega/\phi$ and
$\Omega/\psi$ are primitive, then either $G/\Core_G(H)$ is dihedral
or $\Omega = \Omega/\phi\times\Omega/\psi$ where both $\phi$ and
$\psi$ are maximally fine nontrivial congruences.  It will be
convenient to do some arguments without assuming primitivity.

We begin by observing how $g$-cyles on $\Omega$ relate to
$g$-cycles on $\Omega/\phi$ and $\Omega/\psi$.

\begin{lemma}\label{0}
Pick $\omega\in\Omega$, and suppose the image $\omega_\phi$ of
$\omega$ in $\Omega/\phi$ lies in a $g$-cycle of length $a$, and
the image $\omega_\psi$ of $\omega$ in $\Omega/\psi$ lies in a
$g$-cycle of length $b$.  Then $\omega$
lies in a $g$-cycle of length $\lcm(a,b)$.  
The $g^{\gcd(a,b)}$-orbit of $\omega$ is a union of
$\phi$-equivalence classes, and also is a union of
$\psi$-equivalence classes.  The congruence $\phi^\psi$
nontrivially coarsens $\phi$ unless either $a=b$ or
$\gcd(a,b)=1$.
\end{lemma}

\begin{proof}
Plainly $g^r$ fixes $\omega_\phi$ precisely when $a\mid r$,
and $g^r$ fixed $\omega_\psi$ precisely when $b\mid r$.
Since $g^r$ fixes $\omega$ if and only if $g^r$ fixes both
$\omega_\phi$ and $\omega_\psi$, it follows that $\omega$ is in
a $g$-cycle of length $\lcm(a,b)$.  The $\phi$-equivalence
classes on $\Omega/\phi$ are precisely the $g^a$-cycles;
the $\psi$-equivalence classes are likewise the $g^b$-cycles.
Hence the $g^{\gcd(a,b)}$-orbit of $\omega$ is a union of
$\phi$-equivalence classes, and also a union of
$\psi$-equivalence classes.  Finally, the $\phi$-equivalence
class $\langle g^a\rangle g^i\omega_\phi$ nontrivially intersects
just the $\psi$-equivalence classes
$\langle g^b\rangle g^{ar+i}\omega_\psi$, or in other words the
classes containing elements of the form
$g^{\gcd(a,b)s+i}\omega_\psi$.  Thus, two $\phi$-classes
$\langle g^a\rangle g^{i\omega_\phi}$ and
$\langle g^a\rangle g^{j\omega_\phi}$ become equivalent in
$\phi^\psi$ if and only if $i\equiv j\pmod{\gcd(a,b)}$.
Here $\phi^\psi$ is the trivial coarse partition if
$\gcd(a,b)=1$, and $\phi^\psi=\phi$ if $a\mid b$; in all
other situations, $\phi^\psi$ is a nontrivial coarsening
of $\phi$.
\end{proof}

We split the proof of Theorem~\ref{diamond} into several cases.
In the first case we illustrate our method by using it to
prove Ritt's result (which is needed to verify the `closure'
hypothesis in Lemma~\ref{dl}).

\begin{case1} If $g$ has a single cycle on $\Omega$, and
both $\Omega/\phi$ and $\Omega/\psi$ are primitive, then
$\Omega=\Omega/\phi\times\Omega/\psi$ and both $\phi$
and $\psi$ are maximally fine nontrivial congruences.
\end{case1}

\begin{proof}
Let $a$ and $b$ be the lengths of the cycles of $g$ on
$\Omega/\phi$ and $\Omega/\psi$, respectively.  Since
$\phi$ and $\psi$ have no nontrivial common coarsening,
$\phi\join \psi$ is trivial, so Lemma~\ref{0} implies
$\gcd(a,b)=1$.  Thus $\Omega=\Omega/\phi\times \Omega/\psi$.
If the congruence $\mu$ nontrivially refines $\psi$, then
$\Omega\ne\Omega/\phi\times\Omega/\mu$, so $\mu$ must not
satisfy the same hypotheses as $\psi$; hence
$\phi$ and $\mu$ have a nontrivial common coarsening.
This is not possible if $\Omega/\phi$ is primitive.
\end{proof}

\begin{case2} If $g$ has two cycles on both $\Omega/\phi$
and $\Omega/\psi$, then $\phi\join\psi$ is nontrivial.
\end{case2}

\begin{proof}
Let $\phi'$ be the equivalence relation on $\Omega$ whose two
equivalence classes are the unions of the $\phi$-equivalence
classes comprising the two cycles of $g$ on $\Omega/\phi$.
Then $g$ acts trivially on $\Omega/\phi'$.  Since $g$ has two
cycles on $\Omega$, and each cycle is contained in a
$\phi'$-equivalence class, the $\phi'$-equivalence classes are
precisely the $g$-cycles on $\Omega$.  Since the same holds for
$\psi'$, it follows that $\phi'=\psi'$ is a nontrivial common
coarsening of $\phi$ and $\psi$, so $\phi\join \psi$ is nontrivial.
\end{proof}

\begin{case3} Suppose the $g$-cycles on $\Omega/\phi$ have
lengths $a_1$ and $a_2$, and $\Omega/\psi$ is a $g$-cycle of
length $b$.  Then $\Omega=\Omega/\phi\times\Omega/\psi$, and if
$\Omega/\phi$ and $\Omega/\psi$ are primitive then both
$\phi$ and $\psi$ are maximally fine nontrivial congruences.
congruence.
\end{case3}

\begin{proof}
By Lemma~\ref{0}, the $g$-cycles on $\Omega$ have lengths
$\lcm(a_1,b)$ and $\lcm(a_2,b)$, and consist of $a_1$ and $a_2$
$\phi$-equivalence classes; since $G$ is transitive on $\Omega/\phi$,
every such equivalence class has the same size, so
$\lcm(a_1,b)/a_1=\lcm(a_2,b)/a_2$ and thus $\gcd(a_1,b)=\gcd(a_2,b)$.
Since $\phi\join\psi$ is nontrivial, Lemma~\ref{0} implies
$\gcd(a_1,b)=1$, so $\Omega=\Omega/\phi\times\Omega/\psi$.

Let $\mu$ be a nontrivial congruence refining $\psi$; if
$\phi\join\mu$ is nontrivial, then primitivity of $\Omega/\phi$
implies $\phi\join\mu=\phi$, so $\mu$ is a nontrivial common
refinement of $\phi$ and $\psi$, contradiction.  Thus
$\phi\join\mu$ is trivial, so from Case~2 we know that $g$ acts
cyclically on $\Omega/\mu$.  Then the previous paragraph
implies $\Omega=\Omega/\phi\times\Omega/\mu$, so $\mu=\psi$.
The same argument shows that $\phi$ is also a maximally fine
nontrivial congruence.
\end{proof}

\begin{case4}
Suppose $g$ is an $a$-cycle on $\Omega/\phi$ and a $b$-cycle on
$\Omega/\psi$, but $g$ has two cycles on $\Omega$.  Then either
$\gcd(a,b)>2$, or both
$\gcd(a,b)=2$ and $\Omega=\Omega/\phi\times\Omega/\psi$.
\end{case4}

\begin{proof}
By Lemma~\ref{0}, each $g$-cycle on $\Omega$ has length
$\lcm(a,b)$.  Since $\Omega$ embeds in the $G$-set
$\Omega/\phi\times\Omega/\psi$, which has cardinality $ab$,
it follows that $\gcd(a,b)\ge 2$, with equality if and only if
$\Omega=\Omega/\phi\times\Omega/\psi$.
\end{proof}

Henceforth suppose the situation of Case 4 holds, and suppose
also that $\Omega/\phi$ and $\Omega/\psi$ are primitive.

\begin{case4a}
If $\gcd(a,b)>2$
then $G/\Core_G(H)$ is dihedral of order twice a prime, and the
trivial fine congruence has index $2$ in both $\phi$ and $\psi$.
\end{case4a}

\begin{proof}
By Lemma~\ref{0}, $\phi^\psi$ is nontrivial unless $a\mid b$,
and likewise $\psi^\phi$ is nontrivial unless $b\mid a$.
Thus, primitivity implies $a=b$.  View
$\Omega$ as the edges of a bipartite graph $\Gamma$ whose vertices are
$\Omega/\phi$ and $\Omega/\psi$. By the previous inference, $\Gamma$ is
2-regular.  The classes of $\phi \join \psi$ correspond to the connected
components of $\Gamma$; so if $\Omega/\psi$ is primitive, Gamma is
connected.  Thus $\Gamma$ is a cycle of length $2a$. Now $G/\Core_G(H)
\subseteq \Aut(\Gamma)$, where here we only consider the automorphisms
that preverse $\Gamma$ as a \emph{bipartite} graph. But $\Aut(\Gamma)
\cong D_a$, and $D_a$ acts freely transitively on $\Omega$.  Because
$G/\Core_G(H)$ acts transitively, it is also isomorphic to $D_a$.  Since
$\Omega/\phi$ is primitive, $a$ is prime.
\end{proof}

\begin{case4b}
If $\gcd(a,b)=2$ and $\Omega=\Omega/\phi\times\Omega/\psi$,
then both $\phi$ and $\psi$ are maximally fine nontrivial
congruences.
\end{case4b}

\begin{proof}
Suppose not, and let $\mu\ne\psi$ be a nontrivial
congruence refining $\psi$.  If $g$ has two cycles on $\Omega/\mu$
then Case 3 provides a contradiction.  Thus, $g$ acts as a $c$-cycle
on $\Omega/\mu$.  Lemma~\ref{0} implies $a\mid c$.  Since
$\mu$ refines $\psi$, we also have $b\mid c$.  Thus $c$ is divisible by
$\lcm(a,b)=ab/2$, and $c$ divides $\#\Omega=ab$.  Since $\mu$ is
nontrivial, we have $c\ne ab$, so $c=ab/2$, and thus each
$\mu$-equivalence class has size $2$.  Finally, $\mu^\phi\join\psi$
has precisely two equivalence classes (and so is nontrivial) unless
$b=2$ and $c=a$.  But if $a=c$ then, as above, $G/\Core_G(H)$ is
dihedral.  However, this cannot happen here, because $a$ is not prime.
\end{proof}

\begin{remark}
We stated the diamond lemma (Lemma~\ref{dl}) in terms of
descending chains of groups (with $A$ and $B$ maximal subgroups
of $G$ which meet in $H$);
there is an analogous result in terms of ascending chains
(with $A$ and $B$ minimal overgroups of $H$ in $G$ which
generate $G$).  However, we do not see how to do the proof
this way, since we do not know how to make a congruence that
behaves with regard to refining in the same way that
$\phi^\psi$ behaves for coarsening.
\end{remark}


\section{Data and speculations}

Theorem~\ref{mainthm} is not true if we allow $I$ to be an abelian
group with two orbits (on the $G$-set $\Omega$ of left cosets of $H$
in $G$).  For instance, let $G$ be the group
of permutations of $\F_9$ of the form $x\mapsto\alpha x+\beta$
with $\beta\in\F_9$ and $\alpha^4=1$.  Let $A$ and $B$ be the subgroups
$\{\pm x+\beta:\beta\in\F_9\}$ and $\{x\mapsto\alpha x:
\alpha^4=1\}$.  Then two maximal chains of groups between
$H:=\{\pm x\}$ and $G$ are $H<B<G$ and $H<\{\pm x+\beta:\beta\in\F_3\}<A<G$,
which have different lengths.  However, the abelian group
$\{x+\beta:\beta\in\F_9\}$ has two orbits on $\Omega$.

There are also counterexamples to Theorem~\ref{mainthm} if we allow
$I$ to be a cyclic group with three orbits on $\Omega$.
For instance, let $H$ be an order-$2$ subgroup of $S_4$ which is not
contained in $A_4$; then any cyclic subgroup of order $4$ has three
cycles on $\Omega$, but two maximal chains of groups between $H$ and $G$
are $H<S_3<S_4$ and $H<V_4<D_8<S_4$.

However, computer searches on small groups suggest there are extremely
few examples in these situations.  In fact, it may be that there
are only finitely many finite groups $G$ containing a core-free subgruop
$H$ such that $(G,H)$ does not have the Jordan property, but $G$ contains
an element having precisely three cycles on the coset space $\Omega$.
As yet we have not been able to extend the methods of this paper to
analyze such a situation.




\begin{thebibliography}{99}
\newcommand{\au}[1]{{#1},}
\newcommand{\ti}[1]{\textit{#1},}
\newcommand{\jo}[1]{{#1}}
\newcommand{\vo}[1]{\textbf{#1}}
\newcommand{\yr}[1]{(#1),}
\newcommand{\pp}[1]{#1.}
\newcommand{\pps}[1]{#1;}
\newcommand{\bk}[1]{{#1},}
\newcommand{\inbk}[1]{in: {#1}}
\newcommand{\xxx}[1]{arXiv:#1}


\bibitem{BWZ}
\au{R.~M. Beals, J.~L. Wetherell and M.~E. Zieve}
\ti{Polynomials with a common composite}
submitted for publication.
\xxx{0707.1552 [math.AG]}

\bibitem{BT}
\au{Y.~F. Bilu and R.~F. Tichy}
\ti{The Diophantine equation $f(x)=g(y)$}
\jo{Acta Arith.}
\vo{95}
\yr{2000}
\pp{261--288}

\bibitem{D}
\au{R. Dedekind}
\ti{Ueber Gruppen, deren s\"ammtliche Theiler Normaltheiler sind}
\jo{Math. Ann.}
\vo{48}
\yr{1897}
\pp{548--561}

\bibitem{GTZ}
\au{D. Ghioca, T.~J. Tucker and M.~E. Zieve}
\ti{Intersections of polynomial orbits, and a dynamical Mordell-Lang
 conjecture}
\jo{Invent. Math.},
to appear.
\xxx{0705.1954 [math.NT]}

\bibitem{GTZ2}
\au{D. Ghioca, T.~J. Tucker and M.~E. Zieve}
\ti{Linear relations between polynomial orbits}
in preparation.

\bibitem{J2}
\au{C. Jordan}
\ti{Note sur la th\'eorie des substitutions}
\jo{Giorn. di mat.}
\vo{X}
\yr{1872}
\pp{116}
(= {\OE}uvres, v. I, no. 55, p. 362)

\bibitem{J}
\au{C. Jordan}
\bk{Trait\'e des substitutions et des \'equations alg\'ebriques}
Gauthier-Villars, Paris, 1870.

\bibitem{M}
\au{P. M\"uller}
\ti{Primitive monodromy groups of polynomials}
\inbk{Recent Developments in the Inverse Galois Problem}
385--401,
Amer. Math. Soc., Providence, RI, 1995.

\bibitem{MZ}
\au{P. M\"uller and M.~E. Zieve}
\ti{On Ritt's decomposition theorems for polynomials}
in preparation.

\bibitem{MP}
\au{M. Muzychuk and F. Pakovich}
\ti{Solution of the polynomial moment problem}
submitted for publication.
\xxx{0710.4085 [math.CV]}

\bibitem{N}
\au{E. Netto}
\ti{Zur Theorie der zusammengesetzten Gruppen}
\jo{J. reine angew. Math.}
\vo{78}
\yr{1874}
\pp{81--92}

\bibitem{NST}
\au{P.~M. Neumann, G.~A. Stoy, and E.~C. Thompson}
\bk{Groups and Geometry}
Oxford University Press, Oxford, 1994.

\bibitem{P1}
\au{F. Pakovich}
\ti{On polynomials sharing preimages of compact sets, and related
 questions}
\jo{Geom. Funct. Anal.},
to appear.
\xxx{math/0603452 [math.DS]}

\bibitem{P2}
\au{F. Pakovich}
\ti{On the functional equation $F(A(z))=G(B(z))$, where $A,B$ are
 polynomials and $F,G$ are continuous functions}
\jo{Math. Proc. Cambridge Philos. Soc.},
to appear.
\xxx{math/0605016 [math.CV]}

\bibitem{P}
\au{F. Pakovich}
\ti{Prime and composite Laurent polynomials}
submitted for publication.
\xxx{0710.3860 [math.CV]}

\bibitem{PRY}
\au{F. Pakovich, N. Roytvarf and Y. Yomdin}
\ti{Cauchy-type integrals of algebraic functions}
\jo{Israel J. Math.}
\vo{144}
\yr{2004}
\pp{221--291}
\xxx{math/0312353 [math.CA]}

\bibitem{Pic}
\au{G. Pic}
\ti{On the structure of quasi-Hamiltonian groups}
\jo{Acad. Repub. Pop. Rom\^ane. Bul. \c{S}ti. A.}
\vo{1}
\yr{1949}
\pp{973--979}
%


\bibitem{Ritt}
\au{J.~F. Ritt}
\ti{Prime and composite polynomials}
\jo{Trans. Amer. Math. Soc.}
\vo{23}
\yr{1922}
\pp{51--66}

\bibitem{Zannier}
\au{U. Zannier}
\ti{On a functional equation relating a Laurent series $f(x)$ to
 $f(x^m)$}
\jo{Aequat. Math.}
\vo{55}
\yr{1998}
\pp{15--43}

\bibitem{Z}
\au{M.~E. Zieve}
\ti{Decompositions of Laurent polynomials}
submitted for publication.
\xxx{0710.1902 [math.NT]}

\end{thebibliography}
\end{document}